\documentclass[a4paper,12pt]{article}
\usepackage[top=1in,left=1in,right=1in,bottom=1.5in]{geometry}
\usepackage{latexsym}
\usepackage{epsfig, ecltree, epic, eepic}
\usepackage{enumerate,amsmath,amssymb,dsfont,pifont,amsthm}
\usepackage{xcolor}
\usepackage{graphicx}
\usepackage{stmaryrd}
\usepackage[arrow, matrix, curve]{xy}
\usepackage{comment}
\usepackage{hyperref}
\hypersetup{
    colorlinks,
    citecolor=black,
    filecolor=black,
    linkcolor=black,
    urlcolor=black
}

\theoremstyle{plain}
\newtheorem{theorem}{Theorem}[section]

\newtheorem{lemma}[theorem]{Lemma}
\newtheorem{proposition}[theorem]{Proposition}

\theoremstyle{definition}
\newtheorem{definition}[theorem]{Definition}
\newtheorem{notation}[theorem]{Notation}

\theoremstyle{remark}
\newtheorem{remark}[theorem]{Remark}
\newtheorem{example}[theorem]{Example}

\newcommand{\bbc}{\mathbb{C}}
\newcommand{\bbr}{\mathbb{R}}

\newcommand{\bbp}{\mathbb{P}}
\newcommand{\bbe}{\mathbb{E}}

\newcommand{\bbn}{\mathbb{N}}

\newcommand{\cf}{\mathcal{F}}

\newcommand{\abs}[1]{\left| #1 \right|}
\newcommand{\norm}[1]{\left\| #1 \right\|}

\makeatletter
\newcommand*\Si{\mathpalette\Si@{.5}}
\newcommand*\Si@[2]{\mathbin{\vcenter{\hbox{\scalebox{#2}{$\m@th#1\bullet$}}}}}
\makeatother

\newcommand{\C}{\mathbb{C}} 
 
\newcommand{\R}{\mathbb{R}}

\newcommand{\N}{\mathbb{N}} 
\newcommand{\F}{\mathcal{F}}
\newcommand{\Ap}{\mathcal{A}^+}
\newcommand{\A}{\mathcal{A}}

\newcommand{\M}{\mathcal{M}}

\newcommand{\V}{\mathcal{V}}
\newcommand{\Var}{\text{Var}}
\newcommand{\Pro}{\mathbb{P}}
\newcommand{\Lo}{\mathcal{L}}
\newcommand{\E}{ \mathbb{E}}

\newcommand{\Pre}{\mathcal{P}}

\newcommand{\B}{\mathcal{B}(\R)}

\newcommand{\ls}{\llbracket}

\newcommand{\cad}{\text{c\`{a}dl\`{a}g}}

\newcommand{\zi}{\zeta^{\infty}}
\newcommand{\tX}{\tilde{X}}

\newcommand{\zd}{\zeta^{\Delta}}

\newcommand{\oXn}{H^n}

\newcommand{\rd}{\widetilde{\bbr^d}}

\begin{document}

\allowdisplaybreaks

\title{\bfseries From Markov Processes to Semimartingales}

\author{%
    \textsc{Sebastian Rickelhoff}%
    \thanks{Department Mathematik, Universit\"at Siegen, D-57068 Siegen Germany,
                      \texttt{Sebastian.Rickelhoff@uni-siegen.de}.
                                                                                         }
    \textrm{\ \ and\ \ }
    \stepcounter{footnote}\stepcounter{footnote}\stepcounter{footnote}
    \stepcounter{footnote}\stepcounter{footnote}%
    \textsc{Alexander Schnurr}%
    \thanks{Department Mathematik, Universit\"at Siegen, D-57068 Siegen Germany,
              \texttt{schnurr@mathematik.uni-siegen.de}. }.
    }
\date{\today}

\maketitle
\begin{abstract}
\noindent
In the development of stochastic integration and the theory of semimartingales, Markov processes have been a constant source of inspiration.  
Despite this historical interweaving, it turned out that semimartingales should be considered the `natural' class of processes for many concepts first developed in the Markovian framework.  As an example, stochastic differential equations have been invented as a tool to study Markov processes but nowadays are treated separately in the literature.  Moreover, the killing of processes has been known for decades before it made its way to the theory of semimartingales most recently.  \\
We describe, \textit{when} these and other important concepts have been invented in the theory of Markov processes and \textit{how} they where transferred to semimartingales.  Further topics include the symbol,  characteristics and generalizations of Blumenthal-Getoor indices.  Some additional comments on relations between Markov processes and semimartingales round out the paper. 

\end{abstract}



\section{Introduction}

Brownian motion is \emph{the} starting point for several developments in the theory of stochastic processes. In a way it is in the heart and the center of the universe of processes. 
Brownian motion is a martingale, a L\'evy process, a Gaussian process,  
and finally it is both,  Markovian and a semimartingale. 

Several developments from Brownian motion to semimartingales took a detour via Markov processes. 
Sometimes this was natural and in a way semimartingales were a further generalization coming from Markov processes. 
On other occasions this was not the case: Markov processes have  been at the center of the studies in the field of stochastic processes at a certain time, and, hence, questions which naturally belong to the theory of semimarginales were somehow treated in the `wrong' framework first. 
This is not surprising, since the concept of semimartingales is much younger than the one of Markov processes.  Jarrow and Protter describe it as follows (\cite{jarrowprotter}, p.76): `The beginnings of the theory of  stochastic integration (...) were motivated and intertwined with the theory of Markov processes'.  In the subsequent sections we will see how the Markovian world has influenced the development of semimartinagales and - on the other hand - how some concepts could have been invented directly in the `right' framework. 

We can not retell the whole story of Markov processes (or semimartingales).  We try to describe the most important steps from our point-of-view and give detailed references wherever possible.  Our focus is always on when and how ideas from Markov processes were transferred to semimartingales. 

The purpose of this paper is threefold: we would like to bring together results from Markov processes and semimartingales. Two topics which are nowadays treated separately most of the time but have had a close relationship in their development. Secondly, the transition of the ideas and results from Markov processes to semimartingales might serve as a blueprint in order to transfer more results from one class of processes to another one.  Thirdly, we provide a platform for further research, in particular for younger researchers who are not able to overlook the vast literature on the topics treated here. 

In the subsequent section, we introduce the classes of processes under consideration.  We present a few historical details and fix the notation.  
Let us already mention, that neither the notation nor the classes' names are uniform in the literature,
e.g.  there are various classes which are called `It\^o process'. Furthermore, dealing with `Feller semigroups', some authors use the bounded continuous functions as the basic space, while in the interplay with semimartingales, the continuous functions vanishing at infinity are much more natural. 

In the third section we consider an aspect which is common in Markovian theory: the exictence of a `killing point' or `cemetery'.  
Blumenthal and Getoor describe this intuitively as follows (\cite{blumenthalget}, page 21): `we regard $t \mapsto X_t(\omega)$ as the path (or trajectory) of a particle moving in the space $E$ (...) until it \emph{dies} at which time it is transported to $\Delta$ where it remains forever.' This concept is transferred to the theory of semimartingales.  It is not straight forward to carry this out, since one of the main tools in semimartingale theory is to work locally. This means that one finds a sequence of stopping times (approaching $\infty$) and such that the stopped processes admit certain properties. It would be natural to consider a sequence approaching the killing time of the process in the generalized framework. This is not possible, since a jump to $\Delta$ can happen all of a sudden, that is, the killing time might not be predictable. The idea is to separate the killing into explosions and other kinds of killing (cf. \cite{cheriditoyor} and \cite{fourthchar}).

For Feller processes the infinitesimal generator contains a lot of information concerning the process itself.  For semimartingales the infinitesimal behavior is described by the so called  characteristics. In the fourth section we reconsider the definition of the characteristics of a semimartingale as it is done in the classical literature, moreover we consider  semimartingales with killing which we have introduced in Section 3 and the accompanying need for a new characteristic. In particular, we state some historical facts concerning the semimartingale characteristics together with some, from our point of view, important results on this topic.

In Section 5 we recall the notion of the `symbol' of a stochastic process which generalizes the well-known characteristic exponent of the L\'evy framework. If the generator of a Feller process admits a sufficiently rich domain, it is a pseudo-differential operator. In its Fourier representation, such an operator can be written as the multiplication with the so called symbol.  Jacob {\cite{nielsursprung}} came up with the idea to use a probabilistic formula in order to calculate this symbol, without even writing down the semigroup and the generator.  Subsequently, it was shown that this formula (or slight modifications of it) could be used to generalize the symbol even beyond the Markovian framework and that the symbol still carries information on the process. 

The symbol can be used to generalize the Blumenthal-Getoor indices.  We describe this procedure in the sixth section and present some results on the usability of these indices. 

Stochastic differential equations (SDEs) where introduced to analyze Markov process from a new perspective. However,  stochastic integration, stochastic differentials and SDEs have been generalizes since than.  In fact, there exist interesting classes of SDEs with non-Markovian solutions.  This topic is treated in Section 7. 

Section 8 contains several complementary topics, which we only sketch. They all have in common, that the interplay or connections between Markov processes and semimartingales are investigated.  

Notation: Most of the notation we are using is standard. In the context of semimartingales we mainly follow \cite{jacodshir}. The only difference is that we write $\chi\cdot id$ for the truncation function $h$, where $\chi$ is measurable with compact support and equal to 1 in a neighborhood of zero. A possible way to choose the \emph{cut-off functions} $\chi$ in different dimensions $m\in\bbn$ is as follows: take a one-dimensional cut-off function $\chi:\bbr\to\bbr$ and define for $x\in\bbr^m$: $\tilde{\chi}(x):=\chi(x^{(1)})\cdots \chi(x^{(m)})$ as the product of the one-dimensional cut-off function. Vectors are always meant to be column vectors. The transposed vector or matrix is written as $v'$ or $Q'$.

\section{The Classes of Processes Under Consideration}

Since several classes of processes are important in our investigations, and since some of these classes are not defined in a unique way in the literature, we have decided to fix some definitions and notations in this section. Readers who are familiar with these classes of processes may skip this section and get back to it later if they want to check a certain definition. In the present paper all stochastic processes under consideration take values in $E\subseteq\bbr^d$, where one or two points of no return are added.  Admittedly,  Markov processes could be treated in a more general framework and there have been attempts to introduce more general state spaces for semimartingales as well (cf. \cite{Meti1982}). However, it seems that most of the literature on semimartingales is still restricted to Euclidean space.  The killing of processes will later be treated separately, hence, the processes $(X_t)_{t \geq 0}$  in this section are all conservative, that is, $\bbp(X_t\in\bbr^d)=1$ for every $t\geq 0$. 

A good interim class between Brownian motion on one side and Markov processes respectively semimartingales on the other side are \emph{L\'evy processes}. These are stochastic processes with values in $\bbr^d$ having stationary and independent increments. In addition, it is demanded that the paths are c\`adl\`ag, that is, they are right-continuous with existing left limits.  There is a vast literature on these processes, which are an interesting class on their own right, a building block for more complicated classes of processes and the starting point for many investigations considering Markov processes and semimartingales. Let us mention the monographs \cite{bertoin} and \cite{sato} for details on this famous class of processes.  One result on L\'evy processes is of particular interest for the following sections: It is a well known fact that the characteristic function $\varphi_{Z_t}:\bbr^d\to\bbc$ of such a process $Z=(Z_t)_{t\geq 0}$ being started in $z\in\bbr^d$
can be written as $(t\geq 0)$
\begin{align} \label{charexp}
  \varphi_{Z_t}(\xi)=\bbe^z\left( e^{i(Z_t-z)'\xi} \right) = \bbe^0\left( e^{iZ_t'\xi} \right) =  e^{-t\psi(\xi) }
\end{align}
where 
\begin{align} \begin{split}\label{lkf}
  \psi(\xi)&= -i \ell'  \xi + \frac{1}{2} \xi'Q \xi
      - \int_{\bbr^d\backslash\{0\}} \left(e^{i y'\xi} -1 - i y'\xi \cdot \chi(y) \right) \,  N(dy).
\end{split}\end{align}
Here, $\ell\in\bbr^d$, $Q$ is a positive semidefinite matrix, $N$ the so called L\'evy measure (cf. \cite{sato} (8.2)) and $\chi:\bbr^d\to \bbr$ is a cut-off function.  The function $\psi$ is a continuous negative definite function in the sense of Schoenberg (cf. \cite{bergforst}, Section 7).  Functions of this kind will appear on several occasions subsequently.  The function $\psi$ ist called \emph{characteristic exponent} of the L\'evy process. 

In Section 6 we will deal with a particular subclass of L\'evy processes, namely (one-dimensional strictly) \emph{$\alpha$-stable processes}.  These are L\'evy processes $Z$ with characteristic function
\[
 \varphi_{Z_t}(\xi)=e^{-ct\abs{\xi}^\alpha}
\]
with constants $c >0$ and $0<\alpha\leq 2$. 
Compare in this context \cite{sato} Section 14.

Before we recall the definition of the most important class of processes under consideration in this article, namely semimartingales, we have to reconsider some fundamental classes of processes first:
\begin{align*}
\mathcal{M}^d \; &:= \{ M: M \text{ is a local martingale with values in } \bbr^d \} \\
\mathcal{L}^d \; &:= \{ M \in \mathcal{M}: M_0=0  \} \\
\mathcal{V}^+&:= \{ A: A \text{ is c\'adl\'ag, adapted process with increasing paths and }\\
&\quad \quad \quad \; \; A_0=0 \text{ and values in } \bbr\} \\
\mathcal{V} \; &:= \{ A: A \text{ is c\'adl\'ag, adapted process with paths of finite variation and }\\
&\quad \quad \quad \; \; A_0=0 \text{ and values in } \bbr\} \\
\mathcal{A}^+ &:= \{ A \in \mathcal{V}^+: \E( \text{Var}(A)_\infty ) < \infty \}\\
\mathcal{A} \;&:= \{ A \in \mathcal{V}: \E( \text{Var}(A)_\infty ) < \infty \}
\end{align*}
The index $d$ is omitted if we consider one-dimensional processes. Moreover, for a class of processes $\mathcal{C}$ we denote by $\mathcal{C}_{loc}$ the localized class of $\mathcal{C}$, and by $\mathcal{C}_{ad}$ the class of all processes $X \in \mathcal{C}$ being additive (cf. \cite{vierleute},  (3.16)).
\begin{definition}\label{def:semi}
We call a stochastic process $X= (X_t)_{t \geq 0}$ on the stochastic basis $(\Omega, \mathcal{F}, (\mathcal{F}_t)_{t \geq 0}, \bbp)$ a \emph{semimartingale} if it is of the form 
$$
X=X_0+M+A,
$$
where $X_0$ is finite-valued and $\mathcal{F}_0$-measurable, $M\in \mathcal{L}$, and $A \in \mathcal{V}$. \\
Moreover, we call a $d$-dimensional stochastic process $X_t=(X^{(1)}_t,...,X^{(d)}_t )'$ a \emph{$d$-dimensional semimartingale} if $X^{(i)}$ is a semimartingale for all $i \in \{1,...,d\}$, and we denote by $\mathcal{S}^d$ the class of $d$-dimensional semimartingales.

\end{definition}
By the famous L\'evy-It\^o decomposition (cf. \cite{sato} Chapter 4), every L\'evy process is a semimartingale (with respect to its natural filtration). The notion of semimartingales was introduced by Fisk in \cite{fisk} under the name `quasimartingale' to investigate necessary and sufficient conditions for a stochastic process to posses a decomposition into the sum of a martingale and a process with paths of finite variation. 
However,  the term semimartingale and its application in stochastic integration was first introduced by Meyer in \cite{meyer2} in 1967, where he defined a stochastic process $X$ to be a semimartingale if it is right-continuous and possesses a decomposition
$$
X=M+A
$$
into a martingale $M$ and a stochastic process with paths of finite variation $A$. Indeed, Meyer was inspired by a famous paper \cite{kunitawatanabe} by Kunita and Watanabe who investigated a more general version of It\^o's formula, namely replacing the Brownian motion by martingales which are not necessarily continuous. The non-continuity led to a much more complicated form.  Kunita and Watanabe needed the underlying probability space to be that of a Hunt process, i.e., a strong Markov process which is quasi-left-continuous. \\
In his paper, which was one of four on the topic of stochastic integration,  Meyer was able to leave the Markovian framework behind.  As the attentive reader might have noticed, Meyer's definition as stated above, is not the most general definition as it is known today. This was introduced by Dol\'eans-Dade and Meyer \cite{ddmeyer} in 1970. \\
For a much more detailed look into the development of stochastic integration, semimartingales and mathematical finance until 1970 we refer to the interesting article \cite{jarrowprotter}. 

Most textbooks which introduce semimartingales follow, in some sense, the historical approach as stated above to define semimartingales. 
A different approach was taken by Protter in \cite{protter}.  In contrast to Definition \ref{def:semi}, Protter defined semimartingales to be those processes for which the stochastic integral is continuous (for a more precise definition see Chapter II.2 of \cite{protter}). Indeed, this provides some advantages to the classical procedure since the proofs of some important results are much more intuitive. By the famous theorem of Bichteler and Dellacherie (see Theorem 43 of \cite{protter}), both definitions of semimartingales are equivalent.  Let us mention that there is a Banach space valued theory of semimartingales and stochastic integration (cf. \cite{Meti1982}). 

The second main class of process, we consider are \emph{Markov processes} $(\Omega, \cf, (\cf_t)_{t\geq 0}, (X_t)_{t\geq 0}, (\theta_t)_{t \geq 0} ,\bbp^x)_{x\in\bbr^d}$. We use this term in the sense of Blumenthal-Getoor (cf. \cite{blumenthalget}) this means in particular that the following formula holds
\begin{align} \label{universalmp}
  P_{s,t}^w(x,A) = P_{0,t-s}^{x}(x,A)=:P_{t-s}(x,A)
\end{align}
where $P_{s,t}^w(x,A)$ is the regular version of $\bbp^w(X_t\in A \, | \, X_s=x)$ with $w,x\in\bbr^d$, $s\leq t$ and $A$ is a Borel set in $\bbr^d$. Markov processes of this kind are often called Markov families or universal Markov processes (cf. \cite{bauer}, \cite{niels3}). 
We assume that every universal Markov process is normal, that is,  $\bbp^x(X_0=x)=1$. This is not a severe restriction. 
Furthermore, we consider only $\bbr^d$ as state space, since we are interested in the connection with semimartingales having values in the same space. In Section \ref{sec:killing} additional points outside the Euclidean space are considered.  

As usual, we associate a semigroup $(T_t)_{t\geq 0}$ of operators on $B_b(\bbr^d)$, the bounded Borel-measurable functions,  with every (universal) Markov process by setting
\[
    T_t u(x):= \bbe^x u(X_t), \quad t\geq 0,\; x\in \bbr^d.
\]
For every $t\geq 0$, $T_t$ is a contractive, positivity preserving and sub-Markovian
operator on $B_b(\bbr^d)$. In order to use methods from functional analysis, we need more structure on the underlying space: denoting by $(C_\infty(\bbr^d),\norm{\cdot}_\infty)$ the Banach space of all functions $u:\bbr^d\to\bbr$ which are continuous and vanishing at
infinity, that is,  $\lim_{\norm{x}\to\infty}u(x) =0$. We call $(T_t)_{t\geq 0}$ a \textit{Feller
semigroup} and $(X_t)_{t\geq 0}$ a \emph{Feller process}, if the semigroup is strongly continuous, that is, $\norm{T_tu-u}_\infty\to 0$ for $t\downarrow 0$, and if the following condition is satisfied:
\begin{align}
  T_t:C_\infty(\bbr^d) \to C_\infty(\bbr^d) \text{ for every }t\geq 0.
\end{align}

Without loss of generality we assume all Feller processes we encounter to be c\`adl\`ag (cf. \cite{revuzyor} Theorem III.2.7). This is natural, because we are dealing with semimartingales all the time. The \emph{generator} of the Feller process $(A,D(A))$ is the closed operator given by
\begin{gather}\label{generator}
    Au:=\lim_{t \downarrow 0} \frac{T_t u -u}{t} \qquad\text{for\ \ } u\in D(A)
\end{gather}
where the \emph{domain} $D(A)$ consists of all $u\in C_\infty(\bbr^d)$ for which the limit \eqref{generator} exists uniformly.  A Feller process is called \emph{rich} if $C_c^\infty(\bbr^d)\subseteq D(A)$, that is, the test functions are contained in the domain of the generator. In the literature this property is sometimes called `nice' or (R).

A universal Markov process $X$ is called \emph{Markov semimartingale}, if $X$ is for every $\bbp^x$ a semimartingale (\cite{jacodshir} Definition I.4.21). If it is in addition quasi-left-continuous, it is called \emph{Hunt semimartingale}. In \cite{mydiss} it has been shown that every rich Feller process is a Hunt semimartingale and even an \emph{It\^o process} in the sense of \c{C}inlar, Jacod, Protter and Sharpe \cite{vierleute}, that is, a Markov semimartingale with characteristics (cf.  Definition \ref{def:chars}) of the following form
\begin{align} \begin{split} \label{chars}
  B^{(j)}_t(\omega)&=\int_0^t \ell^{(j)}(X_s(\omega)) \ ds \\
  C^{jk}_t(\omega)&=\int_0^t Q^{jk} (X_s(\omega)) \ ds \\
  \nu(\omega;ds,dw)&=N(X_s(\omega),dw)\ ds
\end{split}\end{align}
for every $x\in\bbr^d$ with respect to a fixed cut-off function $\chi$. Here
$\ell(x)=(\ell^{(j)}(x))_{1\leq j \leq d} \in \bbr^d$, $Q(x)=(Q^{jk}(x))_{1\leq j,k \leq d}$ is a symmetric positive semidefinite matrix, $N(x,dw)$ is a measure on $\bbr^d\setminus\{0\}$ such that $\int_{w\neq 0} (1 \wedge \norm{w}^2) \,N(x,dw) < \infty$. We call $\ell$, $Q$ and $n:=\int_{w\neq 0} (1\wedge \norm{w}^2) \ N(\cdot,dw)$ the \emph{differential characteristics} of the process. 
The class of semimartingales (which are not necessarily Markovian), having characteristics as described in \eqref{chars} are called \emph{homogeneous diffusions with jumps}, h.d.w.j. for short (cf. \cite{jacodshir} Definition 3.2.18). 

L\'evy processes are exactly those Markov processes being homogeneous in time \emph{and} space (having c\`adl\`ag paths) and they are exactly the h.d.w.j. having deterministic differential characteristics. 

For the readers convenience we include the following diagram, parts of which have been presented in various of our papers (\cite{generalizedindices}, \cite{detmp2}, \cite{detmp1}):
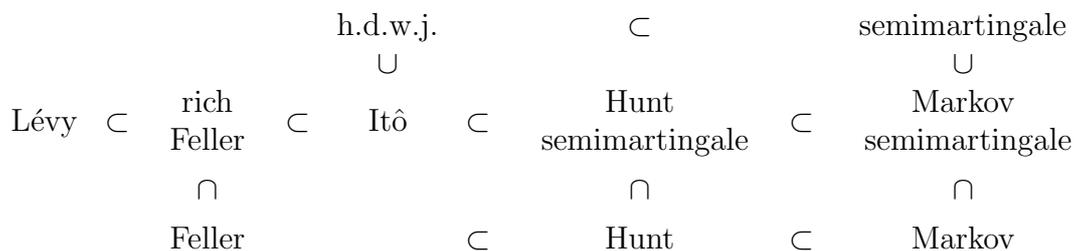
\begin{align*}
\begin{array}{ccccccccc}
&& &&\text{h.d.w.j.}&& \subset && \text{semimartingale}\\
&&  &&\cup&&&& \cup \\
    \text{L\'evy}  & \subset  & \begin{array}{c}\text{rich} \\ \text{Feller} \end{array} & \subset & \text{It\^o}& \subset & \begin{array}{c}\text{Hunt} \\ \text{ semimartingale} \end{array}
         &\subset  &   \begin{array}{c}\text{Markov} \\ \text{ semimartingale} \end{array}   \\ 
                 & \rule[5mm]{0mm}{0mm} &    \cap & &&& \cap &&\cap\\
                 & \rule[5mm]{0mm}{0mm} & \text{Feller} & & &\subset  &\text{Hunt}&\subset&\text{Markov}
\end{array}
\end{align*}
Every inclusion in this diagram is strict. 

\section{Killing of Markov Processes and Semimartingales} \label{sec:killing}
As we have mentioned above,  historically, while exploring the area of semimartingales many researchers took properties of Markov processes as guidelines for new results on semimartingales.  The first topic on which we will emphasize this procedure is the killing of processes. If Markov processes are defined via martingale problems, or sub Markovian kernels, probability mass might be lost over time. In this case one usually adds an ideal point $\partial$ to the state space $E$ and allows for a transition to this \textit{graveyard} or \textit{killing point} (cf. \cite{ethierkurtz} page 166).
In this case, $\partial$ is a point-of-no-return in the sense that, if $X_s(\omega) = \partial$ then $X_t(\omega)=\partial$ for all $t \geq s$. Moreover, the stopping time $\zeta(\omega):= \inf \{ t \geq 0 : X_t(\omega)=\partial \}$ is called the lifetime of the Markov process $X$.  Although, well observed and understood, the concepts of killing points and lifetimes were not integrated into the classical theory of semimartingales for a long time. \\
In 2005 Cheridito, Filipovi\'c and Yor (see \cite{cheriditoyor}) finally dealt with this topic: similar to the Markovian context the authors have considered a stochastic process with values in a state space $(E \cup \{\partial\},\mathcal{E}_\partial)$ where $E$ is a closed subset of $\R^d$. They have set $\|\partial\|:= \infty$ and $T_\partial:= \inf\{ t \geq 0 : X_t= \partial \text{ or } X_{t-}=\partial \}$. The authors have discovered that a transition to $\partial$ occurs either by a jump or by an explosion (see below).  The main idea was to separate the space of paths depending on the kind of killing that occurs.  To this end they have used a sequence of stopping times. 
The processes they have considered are no semimartingales in the usual sense. To turn them into classical semimartingales Cherdito et al.  have separated the killing state like the space of paths.  Here, $\Delta$ denotes the state, which is reached all of a sudden. They have picked an arbitrary point $y$ in $\bbr \backslash E$ and set $\Delta:=y$. If $E=\bbr^d$ even an artificial new dimension was added in order to include one of the points of no return into Euclidean space.  Explosions can not be treated in the same way. Hence, the authors have demanded, that for an announcing sequence of the explosion time, every stopped process is a classical semimartingale.  Summing up: they have dealt with two possible kinds of killing in different ways, but in each case, they got rid of the points that are not in the Euclidean space. Afterwards, they treated the process as in the classical theory of semimartingales.  Using this procedure, the information on killing is lost. One can not write down a representation of the whole process and the new process is not canonical, because the point $y$ can be chosen arbitrarily. 
In \cite{fourthchar} Schnurr has used the idea to separate the process considering the two ways of killing. However, there $\Delta$ remains what it is, namely a point outside the Euclidean space. Explosions can be treated by stopping along an announcing sequence.  A new characteristic is introduced, which describes the sudden killing.  This is in-line with the characteristic exponent of L\'evy processes and allows to treat the generator of Markov semimartingales via (four) characteristics. 

In  the following let $E\subseteq \R^d$ be a closed set, and let $\mathcal{E}$ be the Borel $\sigma$-field on $E$. We equip $E$ with the so-called \textit{killing points}, namely,  $\infty$ and $\Delta$, lying outside $\R^d$. Let $\tilde{E}=E\cup \{\infty\} \cup \{ \Delta \}$ and $\tilde{\mathcal{E}}$ be the $\sigma$-field on $\tilde{E}$ containing $\mathcal{E}, \{\infty\}$, and $\{\Delta\}$.  We consider a $\cad$ stochastic process $X$ on a stochastic basis  $(\Omega, \F, (\F_t)_{t\geq0},\Pro)$ with values in $\tilde{E}$.  Moreover, we define $\F^X_t :=\sigma( X_s : s \leq t )$ for $t \in \R_+$ and $\F^X= \bigvee_t \F^X_t$.  We establish the following calculation rules for the points $\infty$ and $\Delta$:

\begin{definition}
Let $\infty$ and $\Delta$ be as above. Then the following holds:
\begin{itemize}
\item[$\bullet$] $\infty+r = \infty$ and $\Delta +r = \Delta$ for $r \in \R^d$ and $\Delta+\Delta = \Delta$, $\infty+\Delta = \Delta$.
\item[$\bullet$]$\Delta-r=\Delta$ for $r\in \R^d\cup \{\infty\}$ and $\infty-r = \infty $ for $r\in \R^d$.
\item[$\bullet$] $\Delta \cdot r = \Delta $ and $\infty\cdot r = \infty $ for $r \in \R^d\setminus \{0\}$ and $\Delta \cdot 0 = 0 $, $\infty \cdot 0 =0$.
\item[$\bullet$] The Euclidean norm $||\cdot ||$ of $\infty$ and $\Delta$ equals $+\infty$.
\end{itemize}
\end{definition}

\begin{definition}\label{1.1}
Let $(\sigma'_n)_{n\in \N}$ be an increasing sequence of stopping times defined by 
\[
\sigma'_n:= \inf\{ t \geq 0 : \ ||X_t|| \geq n \text{ or } ||X_{t-}|| \geq n \}.
\]
We call $(\sigma'_n)_{n\in\N}$ a \textit{separating sequence}. 
\end{definition}

As the name of the sequence of stopping times indicates it provides a separation between the two different ways in which the process $X$ is able to leave the classical space $E$ and takes the values $\infty$ or $\Delta$.  The next definition shows exactly this:

\begin{definition}\label{1.2}
Let $X$ be a stochastic process with separating sequence $(\sigma'_n)_n$. We define stopping times $\zeta^\partial, \zeta^\Delta, \zeta^\infty$ and $\sigma_n$ as follows
\begin{align*}
\zeta^\partial &:= \inf\{ t \geq 0: \ X_t\in \{\Delta,\infty\}\}, \\
\zeta^\Delta &:= \begin{cases} \zeta^\partial, & \text{if } \sigma'_n = \zeta^\partial \text{ for some } n\in \N\\
\infty,& \text{if }  \sigma'_n < \zeta^\partial \text{ for all } n\in \N \end{cases},\\
\zeta^\infty &:= \begin{cases} \zeta^\partial, & \text{if } \sigma'_n < \zeta^\partial \text{ for all } n\in \N\\
\infty,& \text{if }  \sigma'_n = \zeta^\partial \text{ for some } n\in \N \end{cases},\\
\sigma_n &:= \begin{cases} \sigma_n', & \text{if } \sigma'_n < \zeta^\partial \\
\infty,& \text{if }  \sigma'_n = \zeta^\partial \end{cases}.\\ 
\end{align*} 
\end{definition}

The stopping time $\zeta^\partial$ is the first time the process $X$ leaves $E$, and, moreover, the stopping time $\sigma_n'$ stops the process at the time where its norm exceeds $n$. Since,  in the case where $\zd$ is finite, it coincides with $\sigma_n'$ for one $n$, one can think of $\zd$ as a sudden killing. Moreover, by the definition of $\zi$ the separating sequence never equals $\zeta^\partial$ (also when $\zd$ is finite),  and in this case one can think of $\zi$ as some kind of explosion.  Let us mention that it might appear more canonical to separate between predictable and totally inaccessible killing. In fact, this leads nowhere. The only useful separation in this case is `explosion vs. everything else'.  Indeed, the explosion time is predictable. In order to be more general, we allow a transition from $\infty$ to $\Delta$, that is, we treat processes of the following kind:

\begin{definition}\label{def:killedprocess}
Let $X$ be a stochastic process on $(\Omega, \F^X,(\F^X_t)_{t\geq 0}, \Pro)$ with values in $\tilde{E}$. Moreover, let $\zeta^\infty$ posses the announcing sequence $(\sigma_n\wedge n)_n$ and let $\zeta^\Delta$ be as above.\\
Then $X$ is called a \textit{process with killing}, if 
\[
X 1_{\ls 0, \zeta^\infty \ls} \subseteq E, \; X 1_{\ls \zeta^\infty, \zeta^\Delta \ls}=\infty \text{  and } 
X 1_{\ls \zeta^\Delta, \infty \ls}=\Delta.
\]
Thereby, we set $\lbrack \zeta^\infty(\omega),\zeta^\Delta(\omega)\lbrack = \emptyset$ if $\zeta^\infty(\omega) \geq \zeta^\Delta(\omega)$. In particular, if $\zeta^\Delta =+\infty$ we call $X$ a \textit{process with explosion}.
\end{definition}

As an example for a process of this kind, one can consider the solution of an SDE with locally Lipschitz coefficients. This solution might already have explosion times. This process is then multiplied with another one, which is 1 until an exponentially distributed killing time, sending it to $\Delta$. In this case, a transition from $\infty$ to $\Delta$ is possible.  A process of this kind is even a generalized semimartingale (see below). 

We now adapt the concept of killing to various classes of processes.  Since the main goal is to define semimartingales with killing,  it is useful to define processes with finite variation and martingales with killing also.  The natural way to do so is to demand a process to fulfill the required properties (i.e.  the finite variation or the martingale property) before it leaves the state space $E$.  We will later see that such a definition is what is required for a proper definition of a semimartingale with killing.

\begin{definition}\label{GenSem}
Let $X$ be a process with killing and killing times $\zi,\zd$, and let $(\tau_n)_{n\in\bbn}$ be the announcing sequence of $\zi$.\\
We define for every $n\in \N$ the following stopping time:
$$ \alpha_n:= \tau_n \wedge \zd$$
and call the sequence $(\alpha_n)_{n\in\bbn}$ the \textit{pre-explosion sequence} of $X$.\\
Let $\mathcal{C}$ be a class of processes on a stochastic basis $(\Omega, \F, (\F_t)_{t\geq 0},\Pro)$. We denote by $\mathcal{C}^\dagger$ the set of all processes $X$ with killing on $(\Omega, \F, (\F_t)_{t\geq 0},\Pro)$ and values in $\tilde{E}$ such that $X$ belongs to $\mathcal{C}$ on $\ls 0, \alpha_n \ls$ for all $n \in \N$.
\end{definition}

\begin{definition}
We call a stochastic process $\tX \in \mathcal{S}^\dagger$ a \textit{generalized semimartingale}.  Such a process possesses a decomposition of the form
$$
\tilde{X}_t=X_t +K_t
$$ 
where $(X_t)_{t\geq 0}$ is a process with explosion and $X_t^{\alpha_n-} \in \mathcal{S}$ for all $n \in \N$ and 
$$
(K_t(\omega))_{t\geq 0 }:=(\Delta \cdot 1_{\ls \zd, \infty \ls}(\omega,t))_{t \geq 0}
$$
 is the so called \textit{killing process}.
\end{definition}

\begin{remark}
Without loss of generality we are able to choose $\alpha_n=\zd$
for all $n\in \N$ if $\zd < \zi$.  This is the case since $\tau_n \uparrow \zi$ for $n\to \infty$.
\end{remark}

With that, the following proposition obviously holds true:

\begin{proposition}\label{r6}
The process $\tilde{X}$ with killing and killing times $\zi, \zd$ is a generalized semimartingale if and only if
$$
\tilde{X}_t= \tilde{X}_0 + V_t + M_t ,\quad t\geq 0
$$
for processes $M \in  \Lo^\dagger$ and $ V \in \V^\dagger$.
\end{proposition}

The class of generalized semimartingales is big enough to contain various examples like L\'evy processes with killing, solutions of SDEs with locally Lipschitz coefficients and certain Markov processes defined by sub-Markovian kernels.  Furthermore, this class can be treated in a similar way as classical semimartingales. 

\begin{example} \label{ex:lpwithkilling}
The characteristic exponent $\psi$ of each L\'evy process is a continuous negative definite function in the sense of Schoenberg (cf. \cite{bergforst}, Section 7).  It is a well-known fact in potential theory that each function of this class can be represented in the following way:
\begin{align} \begin{split}\label{lkfa}
  \phi(\xi)&= a -i \ell'  \xi + \frac{1}{2} \xi'Q \xi
      - \int_{\bbr^d\backslash\{0\}} \left(e^{i y'\xi} -1 - i y'\xi \cdot \chi(y) \right) \,  N(dy).
\end{split}\end{align}
That is, the component $a>0$ does not have a stochastic counterpart in the classical L\'evy world. However,  even with the additional component one can associate a stochastic process $\widetilde{Z}$ with this characteristic exponent via \eqref{charexp}. This process is the L\'evy process $Z$ associated with $(\ell,Q,N)$ with the following modification: with $a$ we associate a killing time, which is exponentially distributed with parameter $a$ and independent of $Z$. The new process $\widetilde{Z}$ (with killing) behaves like $Z$, but as soon as the killing time is reached, it jumps to $\Delta$.
In our notation from above, this can be written as $\widetilde{Z_t}=Z_t+K_t$ where $K=(K_t)_{t\geq 0}$ denotes the `killing process' which only attends values in $\{0,\Delta\}$. 
\end{example}

\section{Semimartingale Characteristics} \label{sec:chars}

Markov processes can be analyzed or even characterized by their generator. L\'evy processes are in 1:1 correspondence with their exponential exponent (cf. \cite{sato} Thm.  8.1 and Cor. 11.6).  It is not a surprise, that people felt the demand to possess similar tools like this in the framework of semimartigales.  This lead to the concept of semimartingale characteristics.  All of these objects have in common that they describe the infinitesimal behavior of the stochastic process.  The search for indices or parameters of this kind can be traced back to Kolmogorov \cite{kolmogorov31}. 

Before we come to the definition of the characteristics for a generalized semimartingale, i.e.,  for a process with killing, we reconsider the characteristics of a semimartingale in the classical sense:
let us recall, that a truncation-function $h:\R^d \to \R^d$ is a measurable function with compact support which coincides with the identity in a neighborhood of zero, and that a cut-off function $\chi:\R^d \to \R^d$ is a measurable function with compact support which equals $1$ in a neighborhood of zero.  Although, we work with cut-off functions most the time, we will construct the semimartingale characteristics with the help of truncation function since it is done similarly in most of the literature.
Let $X$ be a $d$-dimensional semimartingale and $h$ a truncation function.  We define two processes as follows:
\begin{align*}
\dot{X}(h)_t&:= \sum_{s \leq t} (\Delta X_s - h(\Delta X_s)), \quad t \geq 0\\
X(h)&:= X- \dot{X}(h).
\end{align*}
A closer look on these processes and the truncation function $h$ provides that the stochastic process $(\Delta X - h(\Delta X))\neq 0$ only if there exists an $\varepsilon>0$ such that $| \Delta X|> \varepsilon $. It follows that $\dot{X}$ is well-defined, since it is the sum of the big jumps of $X$ of which only countable-many exist pathwise.  Moreover,  the process belongs to $\V^d$, and we deduce that $X(h)$ is a $d$-dimensional semimartingale.\\
By observing the jumps of $X(h)$, we easily see that
\begin{align*}
\Delta X(h)= \Delta X - \Delta \dot{X}(h) = h(\Delta X).
\end{align*}
So it follows that $\Delta X(h)$ is bounded since $h$ is. Therefore, $X(h)$ is a special semimartingale, i.e. , a semimartingale with a predictable finite variation part, by Proposition I.4.24 of \cite{jacodshir} and possesses the following canonical decomposition 
\begin{equation}\label{II 2.5}
X(h)=X_0+M(h)+B(h)
\end{equation}
where $M(h) \in \Lo^d$ and $B(h)\in \V^d$ is predictable.
With this, we are now able to define the characteristics of a semimartingale.

\begin{definition}\quad \\  \label{def:chars}
Let $h$ be a truncation function and $X$ be a d-dimensional semimartingale. 
\begin{itemize}
\item[(i.)] We define $B:=(B^1,...,B^d)'$ to be the predictable process $B(h)$ defined in (\ref{II 2.5}).
\item[(ii.)] We define $C:=(C^{ij})_{i,j \leq d}$ to be the continuous process belonging to $\V^{d\times d}$ defined by 
\[
C^{ij} := \langle X^{i,c},X^{j,c} \rangle
\]
for $i,j\in \{1,...,d\}$, where $\langle \cdot, \cdot \rangle$ is the \textit{predictable quadratic covariation} defined in Theorem I.4.2 of \cite{jacodshir}.
\item[(iii.)] We define the predictable random measure $\nu$ on $\R_+ \times \R^d$ to be the \textit{predictable compensator }(see Theorem II.1.8 of \cite{jacodshir}) of the integer-valued random measure $\mu^X$.
\end{itemize}
We call the triplet $(B,C,\nu)$ the \textit{characteristics} of $X$.
\end{definition}

In some sense, the process $B$ describes a predictable `drift' of the semimartingale $X$ through time, $C$ describes the `volatility' of the continuous martingale part $X^{i,c}$, and $\nu$ describes the rate of jumps. This resembles the characteristic triplet of a L\'evy process.  \\ 
\begin{remark}
To our knowledge the first time that semimartinale characteristics have been defined in the modern way was by
Jacod and M\'emin \cite{jacodmem} in 1976.  In their work, the authors defined the characteristics almost as above but only considering the cut-off-function $\chi(x)= 1_{[0,1]}(|x|)$. They investigated how a change of measure effects the characteristics of a semimartingale. \\
The first idea for the characteristics of a semimartingale dates back to Grigelionis \cite{grig} in 1971 or in english languge in \cite{grig2} in 1972.  In order to investigate problems like nonlinear filtering of stochastic processes or absolute continuity of measures corresponding to stochastic processes Grigelionis wanted to consider a wide class of stochastic processes for which one could naturally define local coefficients of drift, diffusion and L\'evy measure. Thus, in his paper Grigelionis considered a $\cad$ process $(X_t)_{t\geq 0}$ with values in $(\R^d, \mathcal{B}(\R^d))$ on a complete stochastic basis $(\Omega, \F, (\F_t)_{t\geq 0}, \Pro)$, a function $\Pi:\R_+\times \Omega \times \mathcal{B}(\R^d)\to \R$ 
where $ \Pi(\omega,t; dx)$ is $\mathcal{B}(\R_+)\otimes \F$-measurable for fixed $(t, \omega)$,  $\Pi(T, \Gamma)$ is $\F_T$-measurable for every $(\F_t)_{t\geq 0}$-stopping time $T$ and every Borel set $\Gamma$. Moreover, 
$$
\E\left( \int_0^t \Pi(s, U_\varepsilon) \ ds \right) < \infty , \E \left( \int_0^t \int_{|x| \leq 1} |x|^2 \Pi(s,dx) ds \right) < \infty, 
$$
where $t> 0, \varepsilon>0$ and 
$$
\sum_{s \leq t} 1_\Lambda(\Delta X_s) - \int_0^t \Pi(s, \Lambda) \ ds
$$
is a right continuous square integrable $(\F_t)_{t\geq 0}$-martingale, and $\Lambda \in \mathcal{B}(\R^d) \cap U_\varepsilon$.\\ 
Moreover Grigelionis considered a $d$-dimensional,  product-measurable function $b_t: (\R^d, \mathcal{B}(\R^d)) \to (\R, \mathcal{B}(\R))$, and a $d\times d$-dimensional,  product-measurable function $c_t: (\R^{d\times d}), \mathcal{B}(\R^{d \times d}) \to (\R, \mathcal{B}(\R))$ with 
$$
\E \left( \int_0^t b_s^2 \ ds \right)<\infty, \quad \E (c_t) < \infty 
$$
and such that 
$$
\tilde{X}_t -  \int_0^t b_s \ ds 
$$
is a continuous, $d$-dimensional, square-integrable martingale, 
and 
$$
 \left\langle \tilde{X}^i_t -  \int_0^t b_s \ ds,\tilde{X}^i_t -  \int_0^t b_s \ ds\right\rangle = \int_0^t c^{i,j}_s \ ds 
$$
with
$$
\tilde{X}_t:= X_t- 2\int_0^t \int_{\{ |x| \leq 1\}} x \ \left( \sum_{u \leq s} \delta_{\Delta X_u}(dx) \right) ds- \int_0^t \int_{\{ |x| \leq 1\}} x \ \left( \int_0^s \Pi(u, dx) \ du \right) ds.
$$
Grigelionis called processes $X$ possessing characteristics as mentioned above \textit{locally infinitely divisible stochastic processes}.
The characteristics of locally infinitely divisible stochastic process $X$ as defined in modern times would be 
\begin{align*}
B_t&= \int_0^t b_s \ ds\\
C_t &= \int_0^t c_s  \ ds\\
\nu(\omega;dt,dx) &= dt \ \Pi(\omega,t;dx).
\end{align*}
Nowadays, processes of this kind are called \textit{It\^o semimartingales}. Let us mention, that estimating the characteristics of such processes by observing the process in a high frequency regime has been a fruitful question in the theory of \textit{statistics of stochastic processes} (cf. \cite{asjac}, \cite{asjac2} and the references given therein). 
\end{remark}

As mentioned above, up to this point we considered the characteristics of a classical semimartingale. But when working with generalized semimartingales we have to redefine the three characteristics as follows.

\begin{definition} 
Let $\tilde{X}$ be a $d$-dimensional generalized semimartingale with killing times $\zi,\zd$ and separating sequence $(\alpha_n)_{n\in \N}$ and let $(B^n,C^n,\nu^n)$ be the characteristics of the semimartingale $X^{\alpha_n-}$. \\
We call the processes $B$ and $C$,  and the random measure $\nu$ the \textit{characteristics} of $\tilde{X}$, if they coincide with  the characteristics $(B^n,C^n,\nu^n)$ of $X^{\alpha_n-}$ on $\ls 0, \alpha_n \ls$ for every $n\in \bbn$. 
\end{definition}

\begin{remark}
\quad \vspace{-5mm}
\begin{itemize}
\item[(a.)]
Since the characteristics are unique up to an evanescent set, and $X^{\alpha_n-}=X^{\alpha_{n+1}-}$ on $\ls 0, \alpha_n \ls$ the characteristics of a generalized semimartingale are well-defined.
\item[(b.)]
By the previous definition, the characteristics of a generalized semimartingale are uniquely defined on $\ls 0, \zd \wedge \zi \ls$ only.  Thus,  we set
\begin{align*}
&C_t(\omega)= C_{(\zd(\omega) \wedge \zi(\omega))-} (\omega)\quad \forall t\geq (\zd \wedge \zi)(\omega), \\
&B_t(\omega)= B_{(\zd(\omega) \wedge \zi(\omega))-}(\omega) \quad \forall t\geq (\zd \wedge \zi)(\omega),\\
&\nu\left(\omega, [\zd\wedge \zi (\omega) , \infty[ \times E \right) =  0 \quad \forall \omega \in \Omega.
\end{align*}
\end{itemize}
\end{remark}

For a classical semimartingale $X$ the random measure $\nu$ compensates the jumps of $X$.  But it is obvious that we are not able to use $\nu$ to compensate a jump with height $\infty$, namely the jump to $\Delta$.  Therefore, we are not able to use only the three characteristics $(B,C,\nu)$ to determine a generalized semimartingale $\tX$.  We are in the need of a new characteristic, and to this end it seems natural not to compensate the jump to $\Delta$ itself, but to keep track, \emph{when} this jump occurs.

\begin{definition}\quad \\
Let $\tilde{X}$ be a generalized semimartingale with values in $\tilde{E}$ and stopping times $\zi,\zd$. We define the process $(A_t)_{t\geq 0}$ to be the predictable compensator of the process $1_{\ls \zd,\infty \ls}$.\\
We call $A$ the \textit{fourth characteristic of a generalized semimartingale}, and, moreover, the quadruple $(A,B,C,\nu)$ the \textit{characteristics of a generalized semimartingale}.
\end{definition}

\begin{definition}\label{def:autosem}
Let $\tX$ be a generalized semimartingale with characteristics $(A,B,C, \nu)$. We call $\tX$ an \textit{autonomous semimartingale} if the characteristics are of the form 
\begin{align*}
A_t&= \int_0^t a(\tX_s) \ ds,   \\
B^i_t &= \int_0^t b^i (\tX_s) \ ds,  \\
C^{ij}_t &= \int_0^t c^{ij}(\tX_s) \ ds,  \\
\nu (\omega;dt,dx)&=\tilde{K}(\tX_t(\omega); dx) \ ds.
\end{align*}
\end{definition}

L\'evy processes with killing (cf. Example \ref{ex:lpwithkilling}) are exactly those autonomous semimartingales having deterministic differential characteristics. 

Having defined the characteristics of a generalized semimartingale, we want to state some of the, in our point of view,  most important results concerning semimartingale characteristics. For the sake of readability,  the proofs of the subsequent results are shifted to the appendix.\\ \quad \\
The characteristics of $\tX$ are unique only up to an evanescent set, and thus it is possible to modify the characteristics on such a set, in order to obtain what we will call the `good' version of $(A,B, C, \nu)$. The following theorem will provide this version, and is one of the main results of this section. Indeed, it is a generalization of Proposition II.2.9 of \cite{jacodshir}.

\begin{theorem}\label{thm:goodversion}
Let $\tX$ be a generalized semimartingale with characteristics $(A',B',C',\nu')$. Then there exists a version $(A,B,C,\nu)$ of $(A',B',C',\nu')$ satisfying the following conditions:
\begin{align}
A_t&= \int_0^t a_s \ dF_s \label{a}\\
B^i_t &= \int_0^t b^i_s \ dF_s \label{b} \\
C^{ij}_t &= \int_0^t c^{ij}_s \ dF_s \label{c}\\
\nu (\omega;dt,dx)&=dF_t(\omega) K_{\omega,t}(dx),\label{K}
\end{align}
where we have
\begin{itemize}
\item[(i.)] a predictable process $F$ belonging to $\Ap_{loc}$,  
\item[(ii.)] a predictable process $a$,
\item[(iii.)]a predictable process $b=(b^1,...,b^d)'$,
\item[(iv.)] a predictable process $c=(c^{ij})_{i,j \leq d}$ taking values in the set of all symmetric, non-negative $d\times d$-matrices,
\item[(v.)] a transition kernel $K_{\omega,t}(dx)$ from $(\Omega\times \R_+, \Pre)$ into $(\R^d, \mathcal{B}(\R^d))$ satisfying 
\begin{itemize}
\item[$\bullet$]$ K_{\omega,t}(\{0\})=0$
\item[$\bullet$] $ \int (|x|^2\wedge 1) \ K_{\omega,t}(dx) \leq 1$
\item[$\bullet$] $\Delta A_t(\omega)K_{\omega,t}(\R^d) \leq 1$.
\end{itemize}
\end{itemize}
Furthermore, the upper `good' version of $(A',B',C',\nu')$ satisfies
\begin{itemize}
\item[(1.)] $(C_t^{ij}-C_s^{ij})_{i,j} \leq d$ is a symmetric non-negative matrix for $s\leq t$.
\item[(2.)] $(|x|^2 \wedge 1) * \nu \in \A_{loc}$ and $\nu(\{t\}\times \R^d) \leq 1$.
\end{itemize}
\end{theorem}

If we combine a generalized semimartingale with the stochastic basis of a strong Markov process we are able to formulate the following stronger version of the above theorem.  It is based on Theorem 6.27 of \cite{vierleute}.

\begin{lemma}\label{lem: ADDLeft}
Let $( \Omega, \mathcal{M}, (\mathcal{M}_t)_{t \geq 0}, (X_t)_{t \geq 0}, (\theta_t)_{t \geq 0}, \Pro^x)_{x \in \R^d}$ 
be a strong Markov process, and let $Y$ be a generalized $\Pro^x$-semimartingale which is additive and quasi-left continuous.  Then the characteristics $(A,B,C, \nu)$ of $Y$ are of the form
\begin{align}
A_t&= \int_0^t a(X_s) \ dF_s \\
B^i_t &= \int_0^t b^i (X_s) \ dF_s  \\
C^{ij}_t &= \int_0^t c^{ij}(X_s) \ dF_s \\
\nu (\omega;dt,dx)&=dF_t(\omega) \tilde{K}(X_t(\omega); dx) 
\end{align}
where 
\begin{itemize}
\item[(i.)] $F$ is continuous, additive, and belongs to $\V^+(\Pro^x)$ for every $\Pro^x$,
\item[(ii.)] $a$ is $\B^d$-measurable,
\item[(iii.)] $b$ is $\B^d$-measurable,
\item[(iv.)] $c$ is $\B^d$-measurable,  $d\times d$- dimensional, and takes values in the set of all symmetric non-negative matrices,
\item[(v.)] $\tilde{K}(\omega,t; dx)$ is a transition kernel from $(\Omega\times \R_+, \mathcal{O}(\mathcal{H}_t))$ into $(\R^d, \mathcal{B}(\R^d))$ with 
 $\tilde{K}(\{0\})=0$ and $ \int (|x|^2 \wedge 1) \; \tilde{K}(dx) < \infty$.
\end{itemize}
\end{lemma}

\begin{theorem}\label{II 2.34}
Let $\tX$ be a $d$-dimensional generalized semimartingale with characteristics $(A,B,C,\nu)$ relative to $\chi$. Then $\tX$ possess the following `canonical' representation:
\begin{align*}
X=X_0+X^c+(\chi\cdot id )*(\mu^X-\nu)+(id(1-\chi))*\mu^X +B+\Delta 1_{\ls \zd, \infty \ls}
\end{align*}
\end{theorem}

\begin{theorem}\label{Schnurr2.7}
Let $\tilde{X}$ be a process with killing  and killing times $\zi,\zd$. Let $(\tau_n)_{n \in \N}$ be the announcing sequence of $\zi$. The following two statements are equivalent
\begin{itemize}
\item[(a.)] $\tilde{X}$ is a generalized semimartingale with killing times $\zi,\zd$ and separating sequence $(\alpha_n)_n$, and characteristics $(A,B,C,\nu)$. 
\item[(b.)]The following processes are local martingales for each $n\in \N$: 
\begin{itemize}
\item[(i.)] $M(h)^{\alpha_n}$, 
\item[(ii.)] $\left(M(h)^jM(h)^k- \tilde{C}^{jk}\right)^{\alpha_n}$ for each $0 \leq j,k \leq d$, 
\item[(iii.)] $\left( g \ast \mu^{\tX^{\alpha_n-}} - g \ast \nu \right)^{\alpha_n}$ for $g \in \mathcal{C}^+(\R^d)$,
\item[(iv.)] $(1_{\ls \zd,\infty \ls}-A)$. 
\end{itemize}
\end{itemize}
\end{theorem}
The proof of the previous two theorems are outlined in Section 2 of \cite{fourthchar}.

\begin{notation}\label{NotL}
\quad \vspace*{-2mm}
\begin{itemize}
\item[(i.)] We denote the vector $(1,...,1)' \in \R^d$ by writing $\mathbf{1}$.
\item[(ii.)] Let $\tilde{X}$ be a generalized semimartingale with values in $\tilde{E}$ and characteristics $(A,B,C,\nu)$. Let $X^n:=\tilde{X}^{\alpha_n-}$ possess the characteristics $(B^n,C^n, \nu^n)$.  We define a complex-valued, predictable process $L^n(u)$ by
\begin{align*}
L^n(u)&:=e^{iu'\mathbf{1}\cdot 1_{\ls \zd, \infty \ls}}\Si A_t  -iu'B^n_t-\frac{1}{2}u'C^n_t u \\
&+ \int_{\R^d} (e^{iu'x}-1-iu'h(x))\ \nu^n([0,t]\times dx).
\end{align*}
\end{itemize}
\end{notation}

\begin{proposition}\label{ÄquiSem}
Let $\tilde{X}$ be a  process with killing, possessing killing times $\zi, \zd$ . Let $X^n$ be defined as before and $H^n:=X^n+\mathbf{1} \cdot 1_{\ls \zd, \infty \ls}$.  Then the following statements are equivalent:
\begin{itemize}
\item[(a.)] $\tilde{X}$ is a generalized semimartingale with characteristics $(A,B,C,\nu)$.
\item[(b.)] The process
\begin{align*}
&e^{iu'H^n}- e^{iu'H^n}\Si A- e^{iu'X^n_-} \Si L^n(u)
\end{align*}
is a complex valued,  local martingale for every $n \in \N$.
\item[(c.)] The process 
\begin{align*}
&f(H^n)-f(X_0)-\sum_{i=1}^d \left(\frac{\partial}{\partial x^i}f(X^n_{-}) \right) \Si (B^n)^i\\
&-\frac{1}{2}  \sum_{i,j=1}^d \left(\frac{\partial^2}{\partial x^i \partial x^j}f(X^n_{-}) \right) \Si  (C^n)^{ij}\\
&-\left[ f(X^n_{-}+x)-f(X^n_{-})- \sum_{i=1}^d \left(\frac{\partial}{\partial x^i}f(X^n_{-}) \right) h((X^n))^i \right]\ast \nu^{n}\\
&-\Delta f(H^n)\Si A
\end{align*}
is a local martingale for every $n \in \N$ and every function $f\in \mathcal{C}^2_b(\R^d)$.
\end{itemize}
\end{proposition}

\section{The Symbol of a Stochastic Process} \label{sec:symbol}

The so called `symbol' of a stochastic process is another concept, which has been invented in the context of Markov processes, and which - in a way - belongs naturally to the theory of semimartingales.  The symbol has proved to be useful in order to derive properties of the process like conservativeness (cf. \cite{schilling98pos}, Theorem 5.5), strong $\gamma$-variation (cf. \cite{sdesymbol} Corollary 5.10), Hausdorff-dimension (cf. \cite{schilling98hdd}, Theorem 4),  H\"older conditions \cite{schilling98}, ultracontractivity of semigroups \cite{schillingwang}, laws of iterated logarithm \cite{schillingknopova} and stationary distributions of Markov processes \cite{behmeschnurr}. 

Again, we do not describe the whole history, but focus on the transition from Markov processes to semimartingales. 

By a classical result due to \ Courr\`ege \cite{courrege} it is known that the generator of a rich Feller process (restricted to the test functions $C_c^\infty(\bbr^d)$) is a \emph{pseudo differential operator}, i.e.,\ $A$ can be written as
\begin{gather} \label{pseudo}
    Au(x)= - \int_{\bbr^d} e^{ix'\xi} q(x,\xi) \widehat{u}(\xi) \,d\xi, \qquad u\in C_c^\infty(\bbr^d)
\end{gather}
where $\widehat{u}(\xi)=(2\pi)^{-d}\int e^{-iy'\xi}u(y) dy$ denotes the Fourier transform. The function $q:\bbr^d \times \bbr^d \to \bbc$ is locally bounded and, for fixed $x$, a continuous negative definite function in the sense of Schoenberg in the co-variable $\xi$ (cf. \cite{bergforst} Chapter II). This is equivalent to the fact that it admits a L\'evy-Khintchine representation
\begin{align} \label{lkfx}
  q(x,\xi)=
  -i \ell'(x)  \xi + \frac{1}{2} \xi'Q(x) \xi -\int_{w\neq 0} \left( e^{i \xi' w} 
  -1 - i \xi' w \cdot \chi(w)\right)N(x,dw)
\end{align}
where  $\ell(x)=(\ell^{(1)}(x),...,\ell^{(d)}(x))'$, $Q(x)$, $N(x,dw)$ are as above in \eqref{chars} and $\chi:\bbr^d\to\bbr$ is a cut-off function. The function $q:\bbr^d\times \bbr^d\to \bbc$ which is often written as $q(x,\xi)$ is called the \emph{symbol} of the operator. For details on the rich theory of the interplay between processes and their symbols we refer the reader to \cite{niels1,niels2,niels3} and for a survey on recent results to \cite{levymatters3}. In the special case of L\'evy processes, the symbol and the characteristic exponent coincide.  In the framework of rich Feller processes, the conditions \eqref{growth} and \eqref{sector} play an important role: 
The growth condition is fulfilled, if there exists a $c>0$ such that
\begin{align} \tag{G} \label{growth}
\norm{q(\cdot,\xi)}_\infty \leq c(1+\norm{\xi}^2)
\end{align}
for every $\xi\in\bbr^d$. The sector condition, which is needed only for some of the results, is fulfilled, if there exists a $c_0>0$ such that for every $x,\xi\in\bbr^d$
\begin{align} \tag{S} \label{sector}
\abs{\Im (q(x,\xi))}\leq c_0 \Re (p(x,\xi))
\end{align}
where $\Re$ resp. $\Im$ denote the real resp. the imaginary part of the function. 

There exists another perspective on the symbol: In \cite{nielsursprung} Jacob came up with the idea to use a probabilistic formula in order to calculate the (functional analytic) symbol defined above:
\begin{align} \label{classicprobsymbol} 
     p(x,\xi):=- \lim_{t\downarrow 0} \frac{\bbe^x e^{i(X_t-x)'\xi}-1}{t}.
\end{align}
The difference to the formulas \eqref{stoppedsymbol} and \eqref{stoppedsymbolkilling} is that no stopping time has been used and no killing was included. In \cite{schilling98pos} the classical formula has been generalized to rich Feller processes satisfying the properties \eqref{growth} and \eqref{sector}. The focus in this work still was, to present a new way to calculate the functional analytic symbol $q(x,\xi)$ in a context, where it already existed.  Leaving the Feller property and \eqref{growth} behind, the symbol for quite general Markov processes was calculated in \cite{mydiss} Chapter 4, namely for It\^o processes. Unlike in earlier papers, the proof relied on the semimartingale structure in particular on the semimartingale characteristics \eqref{chars}.  On the other hand, the earlier results where included, since, as we have mentioned above, every rich Feller process is an It\^o process. The new idea of proof opened the door to define the symbol for general semimartingales (cf. \cite{mydiss}, Definition 4.3):
\begin{definition} \label{def:symbol}
Let $X$ be a stochastic process starting in $x$,  which is conservative, that is, the process does not admit a killing. Fix a starting point $x$ and define $\sigma=\sigma^x_k$ to be the first exit time from a compact neighborhood $K:=K_x$ of $x$: 
  \[
    \sigma:=\inf\{t\geq 0 : X_t^x \notin K \}.
  \]

For $\xi \in \bbr^d$ we call $p:\bbr^d\times \bbr^d\to \bbc$ given by
\begin{align} \label{stoppedsymbol} 
     p(x,\xi):=- \lim_{t\downarrow 0}\bbe^x \frac{e^{i(X^\sigma_t-x)'\xi}-1}{t}
\end{align}
the \emph{(probabilistic) symbol} of the process, if the limit exists and coincides for every choice of $K$.
\end{definition}

\begin{theorem} \label{thm:stoppedsymbol}
Let $X$ be a h.d.w.j. such that the differential characteristics $\ell$, $Q$ and $n$ are continuous.
In this case the limit \eqref{stoppedsymbol} exists and the symbol of $X$ is 
\begin{align}\label{symbol}
  p(x,\xi)=-i\ell(x)'\xi + \frac{1}{2} \xi'Q(x) \xi -\int_{y\neq 0} \Big(e^{iy'\xi}-1 -iy'\xi\cdot\chi(y)\Big) \ N(x,dy).
\end{align}
\end{theorem}

In fact, it is enough that the differential characteristics are locally bounded and finely continuous. Fine continuity has its origins in the framework of Markov processes (cf. \cite{blumenthalget} Section II.4 and \cite{fuglede}). Classical continuity is general enough for all practical purposes. Since we have dealt with killing in the framework of semimartingales already, we can also consider the following - more general - version of the symbol, as it has been introduced in \cite{fourthchar}.

Dealing with the symbol, we could work on $E$ with its relative topology. We make things a bit easier by prolonging the process to $\bbr^d$ by setting $X_t:=x$ for $x\in \bbr^d \backslash E$ and $t\geq 0$. Hence, from now on we assume that our processes live on $\bbr^d$ respectively on $\rd=\bbr^d\cup \{\infty,\Delta\}$. Starting with a process on $E$ local boundedness and fine continuity of the differential characteristic are not harmed by this extension. 
Furthermore, we write for $\xi\in\bbr^d$
\[
  e_\xi(x):=\begin{cases}e^{i x'\xi}& \text{if } x\in \bbr^d \\ 0 &\text{if } x\in\{ \infty, \Delta \}. \end{cases}
\]

\begin{definition} \label{def:symbol}
    Let $X$ be an $\rd$-valued semimartingale, with respect to $\bbp^x$ for every $x\in\bbr^d$. Fix a starting point $x\in\bbr^d$ and let $K\subseteq \bbr^d$ be a compact neighborhood of $x$. Define $\sigma$ to be the first exit time of $X$ from $K$:
    \begin{gather} \label{stopping}
        \sigma:=\sigma^x_K:=\inf\big\{t\geq 0 : X_t\in \rd \backslash K  \big\}.
    \end{gather}
    The function $p:\bbr^d\times \bbr^d \rightarrow \bbc$ given by
    \begin{gather} \label{stoppedsymbolkilling}
         p(x,\xi):= -\lim_{t\downarrow 0}  \frac{\bbe^x\Big(e_\xi(X_t^\sigma-x)  - 1\Big)}{t}   
    \end{gather}
    is called the \emph{(probabilistic) symbol of the process}, if the limit exists for every $x\in \bbr^d$, $\xi\in\bbr^d$ independently of the choice of $K$.
\end{definition}

If we need the symbol on $\rd$, it is defined as follows: in $\Delta$ it is zero and in $\infty$ it is the local killing rate (starting in $\infty$): 
$\lim_{h\downarrow 0} \frac{\bbp^\infty(X_h = \Delta)}{h}$.
The following result is Theorem 2.18 in \cite{fourthchar}. 


\begin{theorem} \label{thm:stoppedsymbolkilling}
Let $X$ be an autonomous semimartingale on $\widetilde{\bbr^d}$ such that the differential characteristics $a$, $\ell$, $Q$ and $n$ are  continuous for every $\bbp^x$ $(x\in \bbr^d)$.
In this case the limit \eqref{stoppedsymbol} exists and the symbol of $X$ is 
\begin{align}\label{symbolkilling}
  p(x,\xi)=a(x) -i\ell(x)'\xi + \frac{1}{2} \xi'Q(x) \xi -\int_{y\neq 0} \Big(e^{iy'\xi}-1 -iy'\xi\cdot\chi(y)\Big) \ N(x,dy).
\end{align}
\end{theorem}

As in other contexts, it was possible to leave the Markov property behind. Indeed, it would have been possible to define the symbol directly for semimartingales, without any knowledge of the connection to the generator in the Markovian framework.  
Coming from the Feller process theory, the formula \eqref{classicprobsymbol} can be interpreted as follows: one plugs the function $y \mapsto e^{i(y-x)'\xi}$ which is bounded, measurable and complex valued into the generator \eqref{generator}. However, without this interpretation, the probabilistic symbol can still be considered a natural object in order to study the behavior of the process: if we forget about the minus sign for a second,  $p(x,\xi)$ is the state space dependent right hand side derivative (in time) of the characteristic function of the process (cf. \eqref{classicprobsymbol}). Since the characteristic function offers a unique way to describe the distribution at a certain point in time,  it is a natural idea to analyze the infinitesimal change of this function in order to derive properties of the process.  In the next section we will see some examples on how the symbol can be used to do this. 

\section{Blumenthal-Getoor Indices}
For $\alpha$-stable processes, there exists a natural index which can be related to different properties of the process (like Hausdorff dimension of the paths, strong variation, ...).  The analysis of the interplay between the stability index $\alpha$ and properties of this kind dates back to Bochner \cite{bochner} and McKean \cite{mckean}.  Having generalized these results to the multivariate framework in \cite{blumenthalget60}, Blumenthal and Getoor \cite{blumenthalget61} introduced in 1961 the indices which were named after them. These indices allowed to analyse more general L\'evy processes.  In \cite{pruitt}, Pruitt introduced another index $\gamma$ which complements the aforementioned indices.  Schilling generalized all of these indices to rich Feller processes satisfying \eqref{growth} and \eqref{sector}. Since the space homogeneity could be left behind, this has probably been the biggest jump in the development of indices of this kind (cf. \cite{schilling98}). 

In order to define the generalized Blumenthal-Getoor indices one uses the following quantities for $x\in\bbr^d$ and $R>0$:
\begin{align}
H(x,R)&:= \sup_{\norm{y-x}\leq 2R} \; \sup_{\norm{\varepsilon}\leq 1} \abs{p\left(y,\frac{\varepsilon}{R}\right)} \\
H(R)&:= \sup_{y\in\bbr^d} \sup_{\norm{\varepsilon}\leq 1} \abs{p\left(y,\frac{\varepsilon}{R}\right)} \\ 
h(x,R)&:= \inf_{\norm{y-x}\leq 2R} \;  \sup_{\norm{\varepsilon}\leq 1} 
\Re \left( p\left(y,\frac{\varepsilon}{4\kappa R} \right)\right) \label{h}\\
h(R)&:= \inf_{y\in\bbr^d} \sup_{\norm{\varepsilon}\leq 1} 
\Re \left( p\left(y,\frac{\varepsilon}{4\kappa R} \right)\right) \label{h2}
\end{align}
In \eqref{h} and \eqref{h2} $\kappa=(4 \arctan (1/2 c_0))^{-1}$ where $c_0$ comes from the sector condition \eqref{sector}.  In particular $h(x,R)$ and $h(R)$ are only defined if \eqref{sector} is satisfied and only in this case they are used. 

As one can see, the definition relies only on the symbol of the process. So, having generalized the symbol to the world of homogeneous diffusions with jumps, it was easy to define these quantities and the related indices in this more general framework. However, it was not clear, whether the close relationship between indices of this kind and properties of the process would remain valid. 

\begin{definition}
The quantities (cf. \cite{schilling98} Definitions 4.2 and 4.5)
\begin{align*}
\beta_0&:=\sup \left\{\lambda \geq 0 : \limsup_{R\to\infty} R^\lambda H(R) =0 \right \} \\
\underline{\beta_0}&:=\sup \left\{\lambda \geq 0 : \liminf_{R\to\infty} R^\lambda H(R) =0 \right \} \\
\overline{\delta_0}&:=\sup \left\{\lambda \geq 0 : \limsup_{R\to\infty} R^\lambda h(R) =0 \right \} \\
\delta_0&:=\sup \left\{\lambda \geq 0 : \liminf_{R\to\infty} R^\lambda h(R) =0 \right \}
\end{align*}
are called \emph{indices of $X$ at the origin}, while
\begin{align*}
\beta_\infty^x&:=\inf \left\{\lambda > 0 : \limsup_{R\to 0} R^\lambda H(x,R) =0 \right \} \\
\underline{\beta_\infty^x}&:=\inf \left\{\lambda > 0 : \liminf_{R\to 0} R^\lambda H(x,R) =0 \right \} \\
\overline{\delta_\infty^x}&:=\inf \left\{\lambda > 0 : \limsup_{R\to 0} R^\lambda h(x,R) =0 \right \} \\
\delta_\infty^x&:=\inf \left\{\lambda > 0 : \liminf_{R\to 0} R^\lambda h(x,R) =0 \right \}
\end{align*}
are the \emph{indices of $X$ at infinity}.
\end{definition}

In the case of symmetric $\alpha$-stable processes all indices coincide and they are equal to $\alpha$. For so called stable-like Feller processes (cf. \cite{bass88a},\cite{negoro94}) with uniformly bounded exponential function, that is, $0<\alpha_0 \leq\alpha(x)\leq\alpha_\infty <1$ one obtains $\beta_0=\underline{\beta_0}=\alpha_0$ and $\delta_0=\overline{\delta_0}=\alpha_\infty$ (see \cite{schilling98} Example 5.5). 

In order to use the symbol and the related indices to derive properties of the process, the following proposition is the key ingredient.  Similar results were proved for L\'evy processes by Pruitt in \cite{pruitt} and for rich Feller processes satisfying \eqref{growth} and \eqref{sector} by Schilling in \cite{schilling98}. The version presented here ist taken from Schnurr \cite{generalizedindices}.  Surprisingly enough, the Markov property was not needed in order to derive results of this kind. We write
\[
(X_\cdot-x)_t^*:=\sup_{s\leq t} \norm{X_s-x}, \quad t \geq 0
\]
for the maximum process.

\begin{proposition} \label{prop:techmain}
Let $X$ be a h.d.w.j. such that the differential characteristics of $X$ are locally bounded and finely continuous. In this case we have
\begin{align} \label{firstestimate}
\bbp^x\Big((X_\cdot - x)_t^* \geq R \Big) \leq c_d \cdot t \cdot H(x,R)
\end{align}
for $t\geq 0$, $R>0$ and a constant $c_d>0$ which can be written down explicitly and only depends on the dimension $d$. \\
If \eqref{sector} holds in addition we have
\begin{align} \label{secondestimate}
\bbp^x\Big((X_\cdot -x)_t^* <R\Big) \leq c_\kappa \cdot \frac{1}{t} \cdot \frac{1}{h(x,R)}
\end{align}
for a constant $c_\kappa$ only depending on the $c_0$ of the sector condition \eqref{sector}.
\end{proposition}

Using this result and standard Borel-Cantelli techniques one obtains the following two theorems which describe the behavior of the process at infinity respectively zero. Let us mention again, that no Markov property is needed in order to derive these results. The whole theory relies nowadays on the semimartingale structure. 

\begin{theorem} \label{thm:ttoinfty}
Let $X$ be a h.d.w.j. such that the differential characteristics of $X$ are locally bounded and finely continuous. Then we have
\begin{align}
\lim_{t\to\infty} t^{-1/\lambda}(X_\cdot-x)_t^*=0 & \text{ for all } \lambda < \beta_0 \\
\liminf_{t\to\infty} t^{-1/\lambda}(X_\cdot-x)_t^*=0 & \text{ for all } \beta_0 \leq \lambda < \underline{\beta_0}.
\end{align}
If the symbol $p$ of the process $X$ satisfies \eqref{sector} then we have in addition
\begin{align}
\limsup_{t\to\infty} t^{-1/\lambda}(X_\cdot-x)_t^*=\infty & \text{ for all } \overline{\delta_0} < \lambda \leq \delta_0 \\
\lim_{t\to\infty} t^{-1/\lambda}(X_\cdot-x)_t^*=\infty & \text{ for all } \delta_0 < \lambda.
\end{align}
All these limits are meant $\bbp^x$-a.s with respect to every $x\in\bbr^d$.
\end{theorem}

\begin{theorem} \label{thm:tto0}
Let $X$ be a h.d.w.j. such that the differential characteristics of $X$ are locally bounded and finely continuous. Then we have
\begin{align}
\lim_{t\to 0 } t^{-1/\lambda}(X_\cdot-x)_t^*=0 & \text{ for all } \lambda > \beta_\infty^x \label{proved}\\
\liminf_{t\to 0 } t^{-1/\lambda}(X_\cdot-x)_t^*=0 & \text{ for all } \beta_\infty^x \geq \lambda > \underline{\beta_\infty^x}.
\end{align}
If the symbol $p$ of the process $X$ satisfies \eqref{sector} then we have in addition
\begin{align}
\limsup_{t\to 0 } t^{-1/\lambda}(X_\cdot-x)_t^*=\infty & \text{ for all } \overline{\delta_\infty^x} > \lambda \geq \delta_\infty^x \\
\lim_{t\to 0 } t^{-1/\lambda}(X_\cdot-x)_t^*=\infty & \text{ for all } \delta_\infty^x > \lambda.
\end{align}
All these limits are meant $\bbp^x$-a.s with respect to every $x\in\bbr^d$.
\end{theorem}

Compare in this context also Fristedt \cite{fristedt}, Taylor \cite{taylor}, Schilling \cite{schillingdiss} and the monograph \cite{levymatters3}.

\section{The Connection to Stochastic Differential Equations}

From our modern point-of-view it is interesting, that It\^o, the father of the theory of stochastic integration, wanted to use stochastic differential equations in order to study Markov processes from a new perspective.  For details we refer to \cite{strookvaradhan} and \cite{jarrowprotter}.  It\^o himself describes it as follows (cf. \cite{itoforeword}): `I noticed that a Markovian particle would perform a time homogeneous differential process for infinitesimal future at every instant, and arrived at the notion of a stochastic differential equation governing the paths of a Markov process that could be formulated in terms of the differentials of a single differential process.'

Modern textbooks on stochastic integration (cf.  \cite{cohenell},\cite{protter}) sometimes do not even mention the term `Markov process' in the first place. Although, this class of processes motivated the development of the theory of stochastic differentials and the corresponding equations, both topics are nowadays taught separately.  To make a long (and interesting) story short, coming from $dW_t$ and $dt$ as the only allowed differentials the integrators became more and more general: martingales, local martingales + Lebesgue measure and finally semimartingales (cf. \cite{jarrowprotter}).  While Kunita and Watanabe \cite{kunitawatanabe} still used the L\'evy system, which was rooted in the theory of Markov processes, in order to describe the jump part in the change-of-variable formula, Meyer (\cite{meyer1} and \cite{meyer2}) left Markov processes and L\'evy systems behind.  This was the point, where what is now known as It\^o's formula left the world of Markov processes and found its place in the correct field of research. Jointly with Dol\'eans-Dade, Meyer coined the modern term semimartingale in \cite{ddmeyer}.  In order to teach people It\^o's formula (see e.g. \cite{protter}) no knowledge on Markov processes is needed. It would have been possible to develop the whole theory of semimartingales,  stochastic integration and SDEs without having the Markov property in mind. 

Dellacherie and Bichteller proved in the late 1970's that semimartingales are the biggest class of stochastic processes with respect to which a reasonable stochastic integration of adapted c\`agl\`ad processes is possible.  Having this characterization, Protter used this approach to actually define semimartingales in \cite{protter} and was able to derive a great amount of results in the theory of stochastic integration relying on that definition. 

Being able to integrate with respect to (vectors of) semimartingales, it is natural to consider SDEs of the form
\begin{align} \label{sde}
X_t=x+\int_0^t f(X_s) \ dY_s
\end{align}
where $Y$ is a semimartingale. For simplicity we only analyze the one-dimensional setting. Let us consider the case, where $f$ is only a function and not a functional of the paths up to time $s$. In this case, it could be shown that the solution - if a unique solution exists - is a Markov process (cf.  \cite{protter} Section V.6)  on a slightly extended probability space, if the driver $Y$ is a L\'evy process.
In \cite{jacodprotter} Jacod and Protter have shown that (if the coefficient $f$  is never zero), this condition is also necessary: if the solution is a Markov process, the driver has to be L\'evy. Since L\'evy process are only a small subclass of general semimartingales, this shows that the canonical case of the solution of an SDE is a process, that is \emph{not} Markovian.  Like the other concepts we have studied above, the theory of SDEs was stimulated by Markov processes, but today the concepts belong to the world of semimartingales. 

Let us shortly comment on the connection between SDEs and the probabilistic symbol of Section \ref{sec:symbol}: if the driver $Y$ in \eqref{sde} is a L\'evy process with characteristic exponent $\psi(\xi)$, and $f$ is so regular that the SDE admits a unique solution, then the symbol of the solution has the nice form: $p(x,\xi)= \psi(f(x)'\xi)$ (cf.  \cite{sdesymbol},  Theorem 3.1).  The condition \eqref{growth} in the theory of the symbol is equivalent to  bounded coefficients in the context of L\'evy driven SDEs (cf. \cite{schilling98} Lemma 2.1 and \cite{jacodshir} Section III.2c).

Metivier deals with different types of (Banach-space valued) SDEs in Chapter 8 of \cite{Meti1982}. Here, the reader also finds some historic remarks.  SDEs driven by semimartingales are treated in \cite{protter} Chapter V in an accessible way.  Compare in this context also \cite{vierleute} Section 8.

\section[Further Connections]{Further Connections Between Markov Processes and Semimartingales}


Lastly, we would like to highlight some other connections between Markov processes and semimartingales. These are not the focus of the present article: they all have in common that it is not ideas and concepts that are transferred from Markov processes to semimartingales. Rather, the connection between these two classes of processes is explored. 

Some authors consider stochastic differential equations with non-Lipschitz coefficients, which are still continuous.  In certain situations it can be shown that solutions of these SDEs exist, but they are not unique. Among the resulting family of solution processes, there might be Markovian ones. In this case one calls the procedure of finding these \emph{Markov selection}.  This problem dates back to a paper due to Girsanov \cite{girsanov} from 1962.  The systematic treatment started with \cite{krylov}. Compare in the context of stochastic partial differential equation \cite{SPDE} and for recent results on L\'evy-type processes \cite{kuehn}. 

In \cite{bentatacont} the authors tackle the following problem: given a discontinuous semimartingale, they want to match a Markov process, whose generator can be expressed in terms of the characteristics of the semimartingale. If such a mimicking process exists, the authors call it \emph{Markov projection}. Obviously, this question is closely related to our considerations above.  Unfortunately, the paper has never appeared in a peer reviewed journal. The latest version on arXiv is from 2012.  If the results of the preprint hold, they generalize a theorem due to Gy\"ongy from 1986 (cf. \cite{gyoengy}). 

In \cite{vierleute} the following question is answered (among many others): Given a Markov process $X$, for which deterministic functions $f$ is the process $f(X)$ a semimartingale? Functions of this kind are called \emph{semimartingale functions} by the authors of that groundbreaking paper.  Loosely speaking the answer for right processes is as follows: $f$ is such a function if and only if it is locally the difference of two excessive functions. For details cf.  Section 4 of \cite{vierleute} and in particular Theorem 4.6. 

In Section 7 of \cite{vierleute} Cinlar et al. deal with the question, when a Markov process is actually a semimartingale. They answer this question for general Markov processes,  as well as Hunt and It\^o processes. The key ingredient is the so called extended generator. 
Let us also mention the nice result in \cite{cinlarjacod81} that every Hunt semimartingale can be obtained as a random time change of an It\^o process.

\section*{Appendix}

Here, we present the proofs omitted in Chapter \ref{sec:chars}.  They give a nice inside into the techniques of proof used in the context of semimartingales with killing and, hence, have some value in their own right. 

\begin{proof}[Proof of Theorem \ref{thm:goodversion}]
Since $B',C',\nu'$ are the characteristics of the classical semimartingale $\tX^{\alpha_n -}$ on $\ls 0, \alpha_n \ls$ there exists a predictable process $F^{'n} \in \A_{loc}$ such that (\ref{b})-(\ref{K}) hold true by Proposition II.2.9 of \cite{jacodshir}.  Since $F^{'n}=F^{'n+1}$ on $\ls 0 , \alpha_n \ls$, (\ref{b})-(\ref{K}) hold for $F':= \lim_{n\to \infty} F^{'n}$.
By defining the process
\begin{equation}\label{F}
F:= F'+ A
\end{equation}
we observe that $dF'\ll dF$ and $dA \ll dF$ since $F'$ and $A$ are increasing.  Moreover, Theorem 3.13 of \cite{jacodshir}  provides the existence of predictable processes $f'$ and $a$ such that 
\begin{align*}
A &= a \Si F, \text{ and} \\
F' &= f' \Si F.
\end{align*}
The associativity of the stochastic integral then provides
\begin{align*}
A&= a \Si F\\
B^i &= (b^if') \Si F  \\
C^{ij} &=(c^{ij}f')\Si F \\
\nu (\omega;dt,dx)&=dF_t(\omega) f'_t(\omega)K_{\omega,t}(dx),
\end{align*}
The properties $(i.)-(2.)$ follow by Proposition II.2.9 of \cite{jacodshir} on $\ls 0, \zi \wedge \zd \ls$. 
\end{proof}

\begin{proof}[Proof of Lemma \ref{lem: ADDLeft}]
At first, we want to prove the additivity of the characteristics:
Let $(\mathcal{M}_t)_{t \geq 0}$ be a Markov filtration, and let $Y$ be additive.  Then $Y$ is additive on $\ls 0 , \alpha_n \ls$ for all $n \in \N$. Thus,  the processes $B_n, C_n$ and $\nu_n$ are additive on said stochastic intervals by Theorem 6.24 (iv) of \cite{vierleute}.  The additivity of $B, C$ and $\nu$ follows. It remains to show that the fourth characteristic, namely $A$, is additive.  Therefore,  we consider the one-point jump process 
$1_{\ls \zd, \infty \ls}$. A simple computations shows that 
\begin{align*}
\Theta_s 1_{\ls \zi, \infty \ls}(\omega,t) &= 1_{\ls s+\zd\circ \theta_s, \infty \ls }(\omega,t)\\
&= 1_{(0,t-s]}(\zd(\theta_s(\omega)).
\end{align*}
Equally easy we state that 
\begin{align*}
1_{\ls \zd \infty \ls }(\omega,t) - 1_{\ls \zd \infty \ls }(\omega,t\wedge s)&= 1_{(s,t]}(\zd(\omega)).
\end{align*}
Since $Y$ is additive it is known that $\zd= \zd\circ \theta_s -s $ for $s < \zd$. Thus, $(A,B,C,\nu)$ are additive.\\
Let us now state the proof of the statement: since $Y$ is quasi-left continuous $B$ is continuous (see I.4.36 and II.2.9 (i.) in \cite{jacodshir}) and $\nu$ is $\Pro^x$ quasi-left continuous (see I.2.35 in \cite{jacodshir}).  $C$ is continuous by definition. Moreover, the quasi left-continuity of $Y$ implies that $\zd$ is totally inaccessible, and, therefore, the fourth characteristic $A$ is continuous.  We are now able to apply Theorem 6.19 of \cite{vierleute}.  Thus, there exists a continuous process $F' \in \V_{ad}^+$ with respect to the strong Markov filtration $(\M_t)_{t \geq 0}$ and a positive transition kernel $\tilde{K}'$ from $(\R^d, \mathcal{B}(\R^d))$ into $(E, \mathcal{E})$, such that 
$$
\nu(\omega; dt;dx) = dF'_t(\omega) \tilde{K}'(X_t(\omega),dx).
$$
Now, let 
$$
F:= F' + \sum_{i\leq d} \Var(B^i) + \sum_{i,j\leq d} \Var(C^{i,j})+ A.
$$
The so defined process $F$ belongs to $\V^+$  and is continuous, and additive.  Moreover,  we have:
$$
dF' \ll dF ,  \; dB^i \ll dF,  \; dC^{ij} \ll dF,  \; dA \ll dF.
$$
Theorem 3.55 of \cite{vierleute} provides the existence of $\B^d$-measurable functions $a,b,c$ such that 
\begin{align*}
B&=b(X) \Si F,\\
C&=c(X) \Si F, \\
A&=a(X) \Si F.
\end{align*}
The theorem follows analogously to the proof of 6.25 in \cite{vierleute}.
\end{proof}

\begin{proof}[Proof of Proposition \ref{ÄquiSem}]
$(a.) \Rightarrow (c.)$: Let $\tilde{X}$ be a generalized semimartingale with characteristics $(A,B,C,\nu)$. We defined $H^n$ to be $X^n+\mathbf{1}\cdot 1_{\ls \zd, \infty\ls}$ where $X^n:= \tilde{X}^{\alpha_n-}$. Thus, the process $H^n$ is a semimartingale since it possesses the decomposition 
$$
H^n=X_0+M^n+ \left(A^n+\mathbf{1}\cdot 1_{\ls \zd, \infty\ls}\right)
$$
where $M^n \in \Lo$ and $A^n+\mathbf{1}\cdot 1_{\ls \zd, \infty\ls} \in \V$ for every $n\in \N$.\\
Let now $h$ be the truncation function belonging to the semimartingale $X^n$ for every $n\in \N$. In order to evaluate the characteristics of $H^n$, we observe that
\begin{align*}
\dot{H^n}(h)_t &:= \sum_{s\leq t} (\Delta H^n_s-h(\Delta H^n_s))\\
&= \sum_{s \leq t < \alpha_n} (\Delta H^n_s-h(\Delta H^n_s))+(\mathbf{1}-h(\mathbf{1}))1_{\ls \zd,\infty \ls} (t)\\
&= \sum_{s\leq t} (\Delta X^n_s-h(\Delta X^n_s))+(\mathbf{1}-h(\mathbf{1}))1_{\ls \zd,\infty \ls} (t)\\
&=\dot{X}^n(h)_t+(\mathbf{1}-h(\mathbf{1}))1_{\ls \zd,\infty \ls} (t),
\end{align*}
and
\begin{align*}
H^n(h)_t&:= H^n_t -\dot{\oXn}(h)_t\\
&= X^n_t+\mathbf{1}\cdot1_{\ls \zd, \infty\ls}(t) - \left(\dot{X}^n(h)_t+(\mathbf{1}-h(\mathbf{1}))1_{\ls \zd,\infty \ls} (t)\right)\\
&=X^n_t-\dot{X}^n(h)_t+h(\mathbf{1})1_{\ls \zd,\infty \ls} (t)\\
&=X^n(h)_t+h(\mathbf{1})1_{\ls \zd,\infty \ls} (t).
\end{align*}
We already know that $X^n(h)$ is a special semimartingale. Thus,  the previous equation shows that $H^n(h)$ also is a special semimartingale with canonical representation
\begin{align*}
H^n(h) &= X^n(h)+h(\mathbf{1})1_{\ls \zd,\infty \ls}\\
&= X_0 + M^n(h)+\left( B^n(h)+h(\mathbf{1})1_{\ls \zd,\infty \ls} \right)
\end{align*}
where  $M^n(h)+B^n(h)+X_0$ is the canonical representation of $X^n(h)$.  This decomposition allows to determine the characteristics $\overline{B^n}$ and $\overline{C^n}$ of $H^n$:
\begin{align*}
\overline{B^n}&:= B^n(h)+h(\mathbf{1})1_{\ls \zd,\infty \ls},\\
\overline{C^n}&:= (\overline{C^n}^{ij})_{i,j \leq d} = \left(\langle (H^n)^{i,c},(H^n)^{j,c} \rangle\right)_{i,j \leq d} = \left(\langle (X^n)^{i,c},(X^n)^{j,c} \rangle\right)_{i,j \leq d} = C^n.
\end{align*}
Analogously to the proof of Theorem II.2.42 of \cite{jacodshir} we apply It\^{o}'s formula to $H^n$ and obtain 
\begin{align*}
f(H^n_t)-f(X_0) &= \sum_{i=1}^d \left(\frac{\partial}{\partial x^i}f(H^n_{-}) \right) \Si \overline{M^n}^i_{t} + \sum_{i=1}^d \left(\frac{\partial}{\partial x^i}f(H^n_{-}) \right) \Si \overline{B^n}^i_{t} \\
&\quad +\frac{1}{2}  \sum_{i,j=1}^d \left(\frac{\partial^2}{\partial x^i \partial x^j}f(H^n_{-}) \right) \Si  \overline{C^n}^{ij}_t \\
& \quad +\sum_{s \leq t} \left[ f(H^n_s)-f(H^n_{s-})- \sum_{i=1}^d \left(\frac{\partial}{\partial x^i}f(H^n_{s-}) \right) h(\left(H^n_s\right)^i )\right]\\
&=\sum_{i=1}^d \left(\frac{\partial}{\partial x^i}f(H^n_{-}) \right) \Si \overline{M^n}^i_{t} + \sum_{i=1}^d \left(\frac{\partial}{\partial x^i}f(H^n_{-}) \right) \Si (B^n)^i_{t}\\
& \quad + \sum_{i=1}^d \left(\frac{\partial}{\partial x^i}f(H^n_{-}) \right) \Si \left(h(\mathbf{1})1_{\ls \zd, \infty \ls}\right) \\
& \quad +\frac{1}{2}  \sum_{i,j=1}^d \left(\frac{\partial^2}{\partial x^i \partial x^j}f(H^n_{-}) \right) \Si  \overline{C^n}^{ij}_t\\
& \quad +\sum_{s \leq t} \left[ f(X^n_s)-f(X^n_{s-})- \sum_{i=1}^d \left(\frac{\partial}{\partial x^i}f(X^n_{s-}) \right) h((X^n_s)^i) \right]\\
& \quad +\Delta f(H^n_{\zd})1_{\ls \zd,\infty \ls}-\sum_{i=1}^d \frac{\partial}{\partial x^i}f(H^n_{\zd-})h(\mathbf{1})1_{\ls \zd,\infty \ls} \\
&=\sum_{i=1}^d \left(\frac{\partial}{\partial x^i}f(H^n_{-}) \right) \Si (M^n)^i_{t} + \sum_{i=1}^d \left(\frac{\partial}{\partial x^i}f(H^n_{-}) \right) \Si (B^n)^i_{t}\\
& \quad + \sum_{i=1}^d \left(\frac{\partial}{\partial x^i}f(H^n_{\zd-}) \right)  h(\mathbf{1})1_{\ls \zd, \infty \ls} \\
& \quad +\frac{1}{2}  \sum_{i,j=1}^d \left(\frac{\partial^2}{\partial x^i \partial x^j}f(H^n_{-}) \right) \Si  (C^n)^{ij}_t\\
& \quad +\sum_{s \leq t} \left[ f(X^n_s)-f(X^n_{s-})- \sum_{i=1}^d \left(\frac{\partial}{\partial x^i}f(X^n_{s-}) \right) h((X^n_s)^i) \right]\\
& \quad +\Delta f(H^n_{\zd})1_{\ls \zd,\infty \ls}-\sum_{i=1}^d \frac{\partial}{\partial x^i}f(H^n_{\zd-})h(\mathbf{1})1_{\ls \zd,\infty \ls} 
\intertext{We now use the fact,  that $B^n_t, M^n_t$ and $C^n_t$ are constant for $t \geq \alpha_n(\omega)$:}
&=\sum_{i=1}^d \left(\frac{\partial}{\partial x^i}f(X^n_{-}) \right) \Si (M^n)^i_{t} + \sum_{i=1}^d \left(\frac{\partial}{\partial x^i}f(X^n_{-}) \right) \Si (B^n)^i_{t}\\
& \quad +\frac{1}{2}  \sum_{i,j=1}^d \left(\frac{\partial^2}{\partial x^i \partial x^j}f(X^n_{-}) \right) \Si  (C^n)^{i,j}_t\\
& \quad +\sum_{s \leq t} \left[ f(X^n_s)-f(X^n_{s-})- \sum_{i=1}^d \left(\frac{\partial}{\partial x^i}f(X^n_{s-}) \right) h((X^n_s)^i) \right]\\
& \quad +\Delta f(H^n_{\zd})1_{\ls \zd,\infty \ls}\\
&=\sum_{i=1}^d \left(\frac{\partial}{\partial x^i}f(X^n_{-}) \right) \Si (M^n)^i_{t} + \sum_{i=1}^d \left(\frac{\partial}{\partial x^i}f(X^n_{-}) \right) \Si (B^n)^i_{t}\\
& \quad +\frac{1}{2}  \sum_{i,j=1}^d \left(\frac{\partial^2}{\partial x^i \partial x^j}f(X^n_{-}) \right) \Si  (C^n)^{i,j}_t\\
& \quad +\left[ f(X^n_{-}+x)-f(X^n_{-})- \sum_{i=1}^d \left(\frac{\partial}{\partial x^i}f(X^n_{-}) \right) h((X^n))^i \right]\ast \mu^{X^n}\\
& \quad +\Delta f(H^n)\Si 1_{\ls \zd,\infty \ls}
\end{align*}
The above equation provides that
\begin{align*}
&f(H^n)-f(X_0)-\sum_{i=1}^d \left(\frac{\partial}{\partial x^i}f(X^n_{-}) \right) \Si (B^n)^i\\
&-\frac{1}{2}  \sum_{i,j=1}^d \left(\frac{\partial^2}{\partial x^i \partial x^j}f(X^n_{-}) \right) \Si  (C^n)^{ij}\\
&-\left[ f(X^n_{-}+x)-f(X^n_{-})- \sum_{i=1}^d \left(\frac{\partial}{\partial x^i}f(X^n_{-}) \right) h((X^n))^i \right]\ast \nu^{n} -\Delta f(H^n)\Si A\\
= &\sum_{i=1}^d \left(\frac{\partial}{\partial x^i}f(X^n_{-}) \right) \Si (M^n)^i\\
&+\left[ f(X^n_{-}+x)-f(X^n_{-})- \sum_{i=1}^d \left(\frac{\partial}{\partial x^i}f(X^n_{-}) \right) h((X^n))^i \right]\ast \left(\mu^{X^n}-\nu^n \right)\\
&+\Delta f(H^n)\Si (1_{\ls \zd,\infty \ls}-A)_t
\end{align*}
Since the right hand side belongs to $\M_{loc}$ the statement follows.\\
$(c.)\Rightarrow (b.)$\\
Let $f:\R^d \to \C; x \mapsto e^{iu'x}$ with $u \in \R^d$.  Obviously,  $f$ is bounded, and belongs to $\mathcal{C}^2(\R^d)$,  and we have 
\begin{align*}
\frac{\partial }{\partial x^j} f(x) = iu^jf(x), \text{ and }\\
\frac{\partial^2 }{\partial x^j\partial x^k} f(x) = -u^ku^jf(x).
\end{align*}
If we compute the expression in $(c.)$ for the function $f$ we obtain that
\begin{align*}
&e^{iu'H^n_t}-\sum_{j=1}^d(iu^je^{iu'X^n_{t-}})\Si (B^n)^j_t- \frac{1}{2}\sum_{j,k=1}^d u^ju^k (e^{iu'X^n_{t-}}) \Si (C^n)_t^{jk}  \\
&- \int_{[0,t]\times \R^d} e^{iu'X^n_{-}}(e^{iu'x}-1-iu'h(x))\ \nu^n(ds\times dx)\\
&-\Delta e^{iu'H^n}\Si A_t\\
= \;&e^{iu'H^n_t}-\sum_{j=1}^d(iu^je^{iu'X^n_{t-}})\Si (B^n)^j_t- \frac{1}{2}\sum_{j,k=1}^d u^ju^k (e^{iu'X^n_{t-}}) \Si (C^n)_t^{jk}  \\
&- e^{iu'X^n_{t-}}\Si \int_{\R^d} (e^{iu'x}-1-iu'h(x))\ \nu^n([0,t]\times dx)\\
&- \left( e^{iu'H^n}-e^{iu'H^n_-}\right) \Si A_t\\
= \;&e^{iu'H^n_t}-\sum_{j=1}^d(iu^je^{iu'X^n_{t-}})\Si (B^n)^j_t- \frac{1}{2}\sum_{j,k=1}^d u^ju^k (e^{iu'X^n_{t-}}) \Si (C^n)_t^{jk}  \\
&- e^{iu'X^n_{t-}}\Si \int_{\R^d} (e^{iu'x}-1-iu'h(x))\ \nu^n([0,t]\times dx))\\
&- \left( e^{iu'H^n}-e^{iu'X^n_-}e^{iu'\left(\mathbf{1}\cdot 1_{\ls \zd, \infty \ls}\right)} \right) \Si A_t\\
= \; & e^{iu'H^n_t}- e^{iu'H^n}\Si A_t- e^{iu'X^n_-} \Si \left(e^{iu'\left(\mathbf{1}\cdot 1_{\ls \zd, \infty \ls}\right)}\Si A_t  -iu'B^n_t-\frac{1}{2}u'C^n_t u \right.\\
&\left.+ \int_{\R^d} (e^{iu'x}-1-iu'h(x))\ \nu^n([0,t]\times dx)\right)
\end{align*}
is a local martingale for every $n\in \N$.\\
$(b.) \Rightarrow (a.)$\\
Let now $e^{iu'H^n_t}- e^{iu'H^n}\Si A_t- e^{iu'X^n_-} \Si L^n(u)$ be a local martingale for every $n \in \N$ and arbitrary $u\in \R^d$.  The process
\begin{align*}
&\left( e^{iu'H^n_t}- e^{iu'H^n}\Si A_t- e^{iu'X^n_-} \Si L^n(u) \right)^{\alpha_n-} &\\
= &e^{iu'X^n}-e^{iu'X^n_-}\Si \left( iu'B^n_t-\frac{1}{2}u'C^n_tu+\int(e^{iu'x}-1-iu'h(x))\ \nu^n([0,t]\times dx) \right)
\end{align*}
is a local martingale.  Application of Theorem II.2.42 of \cite{jacodshir} provides that $X^n$ is a semimartingale with characteristics $(B^n, C^n,\nu^n)$.  Thus,  the generalized semimartingale $\tilde{X}$ possess the characteristics $(B,C,\nu)$.  It remains to show that the process $A$ is the fourth characteristic of $\tilde{X}$. Let therefore be $A'$ the fourth characteristic of $\tilde{X}$. We already know that implication $(a.) \Rightarrow (b.)$ holds.  Let $L'^n(u)$ be the process mentioned in Notation \ref{NotL} with $(A',B^n, C^n,\nu^n)$ such that
$$
e^{iu'H^n_t}- e^{iu'H^n}\Si A'_t- e^{iu'X^n_-} \Si L'^n(u)
$$ 
is a local martingale for every $n \in \N$.
It follows that 
\begin{align*}
&e^{iu'H^n_t}- e^{iu'H^n}\Si A_t- e^{iu'X^n_-} \Si L^n(u) - \left(e^{iu'H^n_t}- e^{iu'H^n}\Si A'_t- e^{iu'X^n_-} \Si L'^n(u)\right)\\
&= \left(e^{iu'H^n}+e^{iu'(X^n_- +\mathbf{1} \cdot 1_{\ls \zd, \infty \ls} )}\right) \Si A' - \left(e^{iu'H^n}+e^{iu'(X^n_-+\mathbf{1} \cdot 1_{\ls \zd, \infty \ls} )}\right) \Si A 
\end{align*}
belongs to $\M_{loc}$ for all $u \in \bbr^d$. Therefore,  $A'-A$ also belongs to $\M_{loc}$. Thus, $A$ is the fourth characteristics of $\tilde{X}$. 
\end{proof}
\section*{Acknowledgment}

We would like to thank Ph. Protter who brought our attention to the interesting article \cite{jarrowprotter}.  AS would like to thank the DFG (German Science Foundation, grant No. SCHN 1231/2-1).

\bibliographystyle{unsrt} 

\end{document}